\documentclass{article}
\usepackage{graphicx}
\usepackage{amsmath}
\usepackage{amsfonts}
\usepackage{amssymb}
\usepackage{ifthen}
\usepackage{xspace}
\usepackage[T1]{fontenc}
\usepackage{psfrag}
\usepackage[english]{babel}

\usepackage[latin1]{inputenc}

\newtheorem{thm}{Theorem}

\newtheorem{defn}[thm]{Definition}
\newenvironment{proof}[1][Proof]{\noindent\textbf{#1.} }{\hfill
  \rule{0.5em}{0.5em} \medskip} 

\newcommand{\R}{\mathbb{R}}
\newcommand{\N}{\mathbb{N}}

\newcommand{\Z}{\mathbb{Z}}

\begin{document}

\title{On Hilbert's 13th Problem
}

\author{Ziqin Feng\footnote{Department
    of Mathematics, University of Pittsburgh, PA 15260, USA} and 
Paul Gartside\footnote{\emph{Corresponding author} Department
    of Mathematics, University of Pittsburgh, Pittsburgh, PA 15260, USA, email:
    gartside@math.pitt.edu.}
}

\date{July 2009}

\maketitle

\begin{abstract} Every continuous function of two or more real variables can be written as the superposition of continuous functions of one real variable along with addition.
\end{abstract}

\section{Introduction}
The 13th Problem from Hilbert's famous list \cite{H} asks whether every continuous function of three variables can be written as a superposition (in other words, composition) of continuous functions of two variables. Hilbert anticipated a negative answer saying, 
\begin{quotation}
``it is probable that the root of the equation of the
seventh degree is a function of its coefficients which [...] cannot be constructed by a finite number of insertions of functions of two arguments. In
order to prove this, the proof would be necessary that the
equation of the seventh degree $f^7 + xf^3 + yf^2 +zf + 1 = 0$ is
not solvable with the help of any continuous functions of only two
arguments.''
\end{quotation}
It took over 50 years  for  significant progress to be made on Hilbert's 13th Problem. Then in 1954 Vitushkin \cite{Vit} found a result in the direction Hilbert expected: if $n/q > n'/q'$ then there are functions of $n$ variables with all $q$th order derivatives continuous which can not be written as a superposition of functions of $n'$ variables and all $q'$th order derivatives continuous. In particular, there are continuously differentiable functions of three variables which can not be written as a superposition of continuously differentiable functions of two variables.

However Kolmogorov and Arnold subsequently proved a series of results culminating  with Kolmogorov's 1957 Superposition Theorem.
\begin{thm}[Kolmogorov Superposition, \cite{Kol}]\label{KST}
For a fixed $n \ge 2$, there are $n(2n+1)$ maps $\psi^{pq} \in
C([0,1])$  such that every map $f \in C([0,1]^n)$ can be
written:
\[ f(\mathbf{x}) = \sum_{q=1}^{2n+1} g_q\left( \phi^q
(\mathbf{x})\right) \qquad \text{where } \phi^q(x_1, \ldots , x_n) = 
\sum_{p=1}^n \psi^{pq} (x_p),\] and the $g_q \in C(\R)$ are maps
depending on $f$. 
\end{thm}
This remarkable theorem gives a very strong positive solution to Hilbert's 13th Problem, indeed it says that every continuous function of two or more variables can be written as a superposition of continuous functions of just one variable along with just one function of two variables, namely addition.

However the Kolmogorov Superposition Theorem is not a complete solution to Hilbert's~13th Problem. Hilbert's statement of the problem explicitly refers to functions (such as the root function of an equation of the seventh degree) of three {\bf real}, or perhaps even more naturally, complex, variables.
But the Kolmogorov Superposition Theorem only deals with functions on a compact cube --- the variables are restricted to a closed and bounded interval.

There have been numerous extensions to the Kolmogorov Superposition Theorem. Most notably Ostrand \cite{Ost} showed that compact, finite dimensional metrizable spaces satisfy a superposition theorem, while Fridman \cite{Frid} showed that the inner functions (the $\psi^{pq}$) can be taken to be Lipschitz. However none  solve Hilbert's 13th Problem for continuous functions of three real variables.

In this paper we complete the solution of Hilbert's 13th Problem by showing that the Kolmogorov Superposition Theorem holds for all continuous functions $f : \R^m \to \R$ (Theorem~\ref{ElemRm}). Further, using earlier work of the authors, \cite{FG1}, we characterize  the topological spaces satisfying a superposition result of the Kolmogorov type. It turns out these spaces  are precisely the locally compact, finite dimensional separable metrizable spaces, or equivalently, those spaces homeomorphic to a closed subspace of Euclidean space (Theorem~\ref{char}).

\section{Superpositions}

Write $C(X,Y)$ for all continuous maps from a space $X$ to another space $Y$, and $C(X)$ for $C(X,\R)$.
Note that we always use the max norm. $\| \cdot \}_\infty$, on $\R^m$.

Abstracting from Theorem~\ref{KST} we make the following definition:
\begin{defn} Let $X$ be a topological space. A family
  $\Phi   \subseteq   C(X)$ is said to be
  basic for $X$ if each  $f \in C(X)$
can be written: \ $f=\sum_{q=1}^{n} \left(g_q \circ \phi_q\right)$,

for some  $\phi_1, \cdots , \phi_n$ in $\Phi$ and
  `co-ordinate functions' $g_1, \ldots , g_n \in C(\R)$.
\end{defn}
Note that the Kolmogorove Superposition Theorem says that every cube $[0,1]^m$ has a finite basic family in which each element of the basic family is a sum of functions of one variable.

\begin{thm}\label{ElemRm} Fix $m$ in $\N$. There exist $\psi^{pq}\in C(\R)$, for $q=1,2, \ldots, 2m+1$ and $p=1,2,\ldots, m$, such that for any function $f\in
C(\R^m)$,  there can be found functions $g_1, \ldots, g_{2m+1}$ in $C(\R)$  such that:
\[f(\mathbf{x})=\sum_{q=1}^{2m+1} g_q(\phi^q(\mathbf{x})), \quad \text{where } \phi^q(x_1, \ldots , x_m)=\psi^{1q}(x_1)+\cdots+\psi^{mq}(x_m).\]
\end{thm}

\begin{proof} We break the proof into four parts. In the first step we define a family of `grids', and approximations to the functions $\psi^{pq}$. Next we define the $\psi^{pq}$ and $\phi^q$, and establish certain useful properties of the grids and functions. In the final two steps  we show that the functions $\phi^q$ are basic for $\R^m$, first for compactly supported functions, and then in general.

\paragraph{1. Construction of the Grids and Approximations}

We establish by induction on $k$, the existence for each $k \in \N$,  $p=1,2,\ldots,m$, and $q=1,2,\ldots, 2m+1$, of positive $\epsilon_{k}$, $\gamma_k<1/10$, 
distinct positive prime numbers $P_k^{p q}>m+10$, discrete
families (`grids') $\mathcal{S}_k^q$ of open intervals of $\R$ and continuous functions $f_k^{p q}: \R\rightarrow \R$ such that:

\begin{itemize}
  \item[(1)]  the sequences of $\epsilon_{k}$'s and $\gamma_k$'s both strictly decrease to zero (in fact, for all $k$, $0< \epsilon_{k+1} < \epsilon_k /6$ and $0 < \gamma_k < 1/k$),
  \item[(2)]  each member of $\mathcal{S}_k^q$ has diameter $\leq \gamma_k$,

  for each fixed $k$ any two of the families $\{\mathcal{S}_k^q: q=1,\ldots,
  2m+1\}$ cover $[-k,k]$, and all cover $\{-k, 0, k\}$;

  \item[(3)]  $m\epsilon_k<1/\prod_{p=1}^m P_k^{p q}$ for
  each $q=1, 2, \ldots, 2m+1$;
  \item[(4)]  $ f_k^{p q}$ is non--decreasing on $\R^{+}$, non--increasing on $\R^{-}$ and constant outside $[-k,k]$;

   \item[(5)]  $ f_k^{p q}$  is  constant on each
  member of $\mathcal{S}^{q}_k$ with value a positive integral multiple of
  $1/P_k^{p  q}$, and $(f_k^{pq}(J_1)-f_k^{pq}(J_2)) P_k^{pq} \mod P_k^{pq}\neq 0$ given $J_1, J_2\in  \mathcal{S}_k^{pq}$;

additionally, if $J$ is an interval containing $0$, then  $f_k^{pq}$ maps $J$ to $0$;

  \item[(6)] $|f_k^{pq}(k)-k|<1/(m+1)$ and $|f_k^{pq}(-k)-k|<1/(m+1)$;

  \item[(7)] for each $\ell \le j<k$ and $x\in[-\ell, \ell]$, $f_{j}^{p q}(x)\leq
  f_{k}^{pq}(x)\leq f_{j}^{p q}(x)+\epsilon_j-\epsilon_k$.
\end{itemize}

\begin{description}
\item[Base Step:]

  It is straightforward to find discrete collections of open intervals $\mathcal{S}_1^{pq}$ for $p=1,  \ldots, m$ and  $q=1, \ldots, 2m+1$
 such that any two of the families $\{\mathcal{S}_1^{pq}: q=1,2,\cdots,
 2m+1\}$ cover $[-1,1]$,  each of the families covers $\{1,0,
 -1\}$, and each interval in the collection has length $\leq
 \gamma_1=1/10$.

  Let $n_1$ be the number of all the open interval in all the
 collections $\mathcal{S}_1^{pq}$ ($1 \le p \le m$,  $1 \le q \le 2m+1$).
For $p=1, \ldots , m$ and $q=1, \ldots , 2m+1$ pick distinct primes $P_1^{pq}$ larger than $n_1$. 

Now we define $f_1^{pq}$ on $[-1,1]$. Then for $x>1$ define $f_1^{pq}(x)=f_1^{pq}(1)$, and for $x< -1$ define $f_1^{pq}(x)=f_1^{pq}(-1)$.

 If $J\in \mathcal{S}_1^{pq}$, then define $f_1^{pq}$ such that $f_1^{pq}$ restricted to $J$ is a positive
 integral multiple of $ 1/P_1^{pq}$. More specifically, if $0\in J$ then $f_1^{pq}(J)=0$;  if $1\in J$ then $f_1^{pq}(J)=1-1/P_1^{pq}$; and if  $1\in J$ then  $f_1^{pq}(J)=1-2/P_1^{pq}$. This can easily be done so that $f_1^{pq}$ (as defined so far) is non--decreasing on $[0,1]$ and non--increasing on $[-1,0]$.

For $x$  in $[-1,1] \setminus \bigcup \mathcal{S}_1^{pq}$,  interpolate $f_1^{pq}$ linearly.

Choose $\epsilon_1>0$ such that $m\epsilon_1<1/\prod_{p=1}^m P_1^{p q}$ for
  each $q=1, 2, \cdots, 2m+1$.

All (applicable) conditions (1)--(7) hold.

\item[Inductive Step:] Suppose $P_{k-1}^{p q}$, $\epsilon_{k-1}$, $\gamma_{k-1}$, $\mathcal{S}_{k-1}^q$ and
$f_{k-1}^{p q}$ are all given and satisfy the requirements
(1)--(7).

By uniform continuity of  $f_{k-1}^{pq}$ on $[-(k-1),k-1]$, there exists
$\gamma_k<\min\{1/k,\gamma_{k-1}\}$ such that
$|f_{k-1}^{pq}(x_1)-f_{k-1}^{pq}(x_2)|<\epsilon_{k-1}/6$ if
$|x_1-x_2|<\gamma_k$ for each $p=1,\ldots , m$ and $q=1,\ldots, 2m+1$.

Then it is straightforward to find discrete collections of open intervals,  
$\mathcal{S}_k^{pq}$ for  $1 \le p \le m$ and  $1 \le q \le 2m+1$,
 such that any two of the families $\{\mathcal{S}_k^{pq}: q=1,2,\cdots,
 2m+1\}$ cover $[-k,k]$, each of the families covers $\{k,0,
 -k\}$, each interval in the collection has length $\leq
 \gamma_{k}$ and the distance between each pair of adjacent intervals is also $\leq
 \gamma_{k}$.

Let $n_k$ be the total number of  open intervals in all the
 collections $\mathcal{S}_k^{pq}$ for $ p=1, 2, \ldots, m$  and $q=1, 2, \ldots,
 2m+1$.
 For each $p, q$ select distinct primes $P_k^{pq}$ so that  $2n_k/P_k^{pq}<\epsilon_{k-1}/6$.

Next, we give the construction of $f_k^{pq}$ on $[-k,k]$. Outside of
$[-k,k]$ extend constantly (as in the Base Step).

\begin{itemize}
\item If $J\in \mathcal{S}_k^{pq}$, then $f_k^{pq}(J)$ is a positive
 integral multiple of $ 1/P_k^{pq}$. For any $ J\in
 \mathcal{S}_k^{pq}$ with $J\cap [-(k-1),k-1]\neq \emptyset$, we can ensure that 
  $f_{k-1}^{pq}(x)<f_k^{pq}(x)<f_{k-1}^{pq}(x)+\epsilon_{k-1}/3$.
  \subitem[i] Since $2n_k/P_k^{pq}<\epsilon_{k-1}/6$ and $|f_{k-1}^{pq}(x_1)-f_{k-1}^{pq}(x_2)|<\epsilon_{k-1}/6$
  when
$|x_1-x_2|<\gamma_k$, there are $2n_k$ possible choices for
the value of $f_k^{pq}(J)$ ($J\in\mathcal{S}_k^{pq}$) which makes
$f_{k-1}^{pq}(x)<f_k^{pq}(x)<f_{k-1}^{pq}(x)+\epsilon_{k-1}/3$ for
$x\in J\cap [-(k-1),k-1]$. As there are many fewer than $2n_k$ elements in
$\mathcal{S}_k^{pq}$, we can  select
the $f_k^{pq}(J)$'s such that
   $(f_k^{pq}(J_1)-f_k^{pq}(J_2))P_k^{pq}\mod P_k^{pq}\neq 0$ for any $J_1, J_2\in
  \mathcal{S}_k^{pq}$.
 \subitem[ii] More specifically, if $0\in J$ then $f_k^{pq}(J)=0$, if  $k\in J$ then $f_k^{pq}(J)=1-1/P_k^{pq}$, and if $-k\in J$ then $f_k^{pq}(J)=1-2/P_k^{pq}$. This can easily be done to make $f_k^{pq}$ (as defined so far) non--decreasing on $[0,k]$ and non-increasing on $[-k,0]$.
 \item If $x\notin \bigcup \mathcal{S}_k^{pq} $, let $J_L$ and $J_R$ be the adjacent intervals in
 $\mathcal{S}_k^{pq}$ such that $x$ lies between them. Let $x_L$ be the
 right endpoint of $J_L$ and $x_R$ be the left end point of $J_R$
 Then $f_k^{pq}$ maps $[x_L, x_R]$ linearly to
 $[f_{k-1}^{pq}(J_L),f_{k-1}^{pq}(J_R)]$. Since $|x_L-x_R|<\gamma_k$,
 $|f_{k-1}^{pq}(x_L)-f_{k-1}^{pq}(x_R)|<\epsilon_{k-1}/6$, therefore,
  $f_k^{pq}(x)-f_{k-1}^{pq}(x)<\epsilon_{k-1}/3+\epsilon_{k-1}/6=\epsilon_{k-1}/2$.
\end{itemize}
Choose $\epsilon_k$ such that
$m\epsilon_k<\min\{1/\prod_{p=1}^m P_k^{p q},\epsilon_{k-1}/6\}$ for
  all $1 \le q \le 2m+1$.

All  requirements (1)--(7) are  satisfied.

\end{description}

\paragraph{2. Definition and Useful Properties of the  Functions, $\psi^{pq}$ and $\phi^q$}

For $x\in \R$, let $\psi^{p q}(x)=\text{lim}_{k\rightarrow \infty}
f_{k}^{p q}(x)$. Now for a fixed $n\in \N$, and any $x\in [-n,n]$,
$f_{k}^{p q}(x)\leq \psi^{p q}(x) \leq f_{k}^{p q}(x)+\epsilon_k$ for $k>n+1$.
So $\psi^{p q}$ restricted to $[-n,n]$, being the uniform limit of the 
$f_k^{p q}$ for $k > n+1$, is continuous on $[-n,n]$. Therefore, $\psi^{p q}$ is
continuous on $\R$.

Also, by construction, the image of $[n,n+1]$ under $\psi^{p q}$ is a subset of
$[|n|-1/(m+1), |n|+1+1/(m+1)]$ for each $n\in \Z$.

\smallskip

Let $\phi^{q}(x_1,\ldots,x_m)=\psi^{1q}(x_1)+\cdots+\psi^{mq}(x_m)$
for $(x_1,x_2,\ldots,x_m)\in \R^m$.

\smallskip

Our eventual goal is to  show
$\{\phi^q: q=1,2,\ldots, 2m+1\}$ is a basic family of $\R^m$, however first, we establish some useful  properties of the grids and functions.

For each $q$ and $k$, let $\mathcal{J}_k^q=\{C_1\times
C_2\times\cdots\times C_m: C_{p}\in \mathcal{S}_k^{q} \text{ for each }
p=1,2,\ldots, m\}$. Then we can say the following about
$\mathcal{J}_k^q$.
\begin{itemize} \item For a fized $q$ and $k$,
$\mathcal{J}_k^q$
 is a discrete collection.

 \item For a fixed $k$, any element in $\R^m$ belongs to at least
 $m+1$ rectangles of $\mathcal{J}_k^q$, i.e. any $m+1$ of $\{\mathcal{J}_k^q: q=1,\ldots, 2m+1\}$ form an open cover of $\R^m$.
\end{itemize}

 Let $\mathcal{U}_{k}^q=\{\phi^q(C):C\in\mathcal{J}_k^q\}$.
 Take $C=C_1\times C_2\times  \cdots\times C_m\in \mathcal{J}_k^q$, then
 $\phi^q(C)$ is contained in the interval $[\sum_{p=1}^m f_k^{pq}(C_p),
 \sum_{p=1}^m f_k^{pq}(C_p)+m\epsilon_k]$.
 By  condition~(3) in the construction of the $f_k^{p q}$,
 these closed intervals are disjoint for each $q$ and $k$.
 Therefore,
\begin{description}
\item[Claim] $\mathcal{U}_{k}^q$ is a discrete collection of subsets of
 $\R$ for each $q$ and $k$.
\end{description}




\paragraph{3. The  $\phi^q$ are Basic for Compactly Supported Functions}
We now prove:

\begin{description}
\item[Claim]  For any compactly supported  $h\in
C(\R^m)$, there are  $g_1, \ldots , g_{2m+1}$ in $C(\R)$ 
such that $h=\sum_{q=1}^{2m+1} g_q \circ \phi_q$.
\end{description}

Fix a compactly supported $h \in C(\R^m)$. Choose $\ell$ in $\N$ so  that
 $h(\mathbf{x})=0$ for any $\mathbf{x}$ outside $K=[-\ell-1, \ell+1]^m$.

For each integer $r\geq 0$ and $q=1,\cdots, 2m+1$, find
positive $k_r$ and continuous functions $\chi_r^q:\R\rightarrow \R$
($k_0=\ell$ and $\chi_1^q=0$ for each $q$) such that if
$h^r(\mathbf{x})=\sum_{q=1}^{2m+1}\sum_{s=0}^r\chi_s^q(\phi^q(\mathbf{x}))$
 and $M_r=\text{sup}_{\mathbf{x}\in\R^m}|(h_i-h^r_i)(\mathbf{x})|$, then:
\begin{itemize}
\item[(1)] $k_{r+1}>k_r$;
\item[(2)] if $\|\mathbf{a}-\mathbf{b}\|_\infty <m/10^{k_{r+1}}$, then
$|(h-h_r)(\mathbf{a})-(h-h_r)(\mathbf{b})|<(2m+2)^{-1}M_r$ for
$\mathbf{a},\mathbf{b}\in \R^m$;
\item[(3)] $\chi_{r+1}^q$ is constant on each member of
$\mathcal{U}_{k_{r+1}}^q$;
\item[(4)] if $C\cap (\R^m\setminus K)\neq \emptyset$ for $C\in \mathcal{J}_{k_{r+1}}^q$, then the value of
$\chi_{r+1}^q$ on $\phi_q(C)$ is $0$,

otherwise, its value on $\phi_q(C)$ is $(m+1)^{-1}(h-h_r)(\mathbf{y})$
for some arbitrarily chosen element $\mathbf{y}\in C$; and

\item[(5)] $\chi_{r+1}^q(x)\leq (m+1)^{-1}M_r$ for each $x\in \R$.
\end{itemize}

The $k_r$ and $\chi_r^q$ are defined inductively on $r$. Also for
any $\mathbf{a},\mathbf{b}\in C\in \mathcal{J}_{k_{r+1}}^q$,
$\|\mathbf{a}-\mathbf{b}\|_\infty <m/10^{k_{r+1}}$.
Therefore:
\begin{itemize}
\item[(6)]  for $\mathbf{x}\in\bigcup\{ C: C\in
\mathcal{J}_{k_{r+1}}^q\}$,

$|(m+1)^{-1}(h-h_r)(\mathbf{x})-\chi_{r+1}^q(\phi_q(\mathbf{x}))|<(m+1)^{-1}(2m+2)^{-1}M_r$.
\end{itemize}

Also for each $\mathbf{x}\in \R^m$, there are at least $m+1$ distinct
values of $q$ such that $\mathbf{x}\in \bigcup\{C:C\in
\mathcal{J}_{k_{r+1}}^q\}$. Then there are $m+1$ values of $q$ such
that (6) is true; for the other $m$ values of $q$, (5) in the
construction holds.

Hence, for $\mathbf{x}\in K$,
\begin{eqnarray*}
|(h-h_{r+1})(\mathbf{x})| &= & |(h-h_r)(\mathbf{x})-\sum_{q=1}^{2m+1}\chi_{r+1}^q(\phi_q(\mathbf{x}))| \\
 &< & (m+1)\cdot (m+1)^{-1}(2m+2)^{-1}M_r +m\cdot (m+1)^{-1} M_r \\
 &= & \frac{2m+1}{2m+2} M_r.
\end{eqnarray*}

While for $\mathbf{x}\notin K$,
$\sum_{q=1}^{2m+1}\chi_{r+1}^q(\phi_q(\mathbf{x}))=0$ by property (4).

Therefore, $M_{r+1}<(2m+1)\cdot (2m+2)^{-1}\cdot M_r$, so
$M_r<((2m+1)\cdot (2m+2)^{-1})^r\cdot M_0$ for each $r$, hence
$\lim_{r\rightarrow \infty} M_r=0$, and thus
$h(\mathbf{x})=\text{lim}_{r\rightarrow\infty}h_r(\mathbf{x})$ for all
$\mathbf{x}\in \R^m$.

Moreover, by condition (5), the functions $\sum_{s=0}^r\chi_s^q$
converge uniformly for each $q$ to a continuous function
$g_q:\R\rightarrow \R$ and
$$h(\mathbf{x})=\text{lim}_{r\rightarrow \infty} h_r(\mathbf{x})=\text{lim}\sum_{q=1}^{2m+1}
\sum_{s=0}^r\chi_s^q(\phi_q(\mathbf{x}))=\sum_{q=1}^{2m+1}g_q(\phi_q(\mathbf{x})).$$

This complete the proof of the Claim.

\paragraph{4. The $\phi^q$ are Basic for All Functions}
We complete the proof by showing:
\begin{description}
\item[Claim]  For any   $f\in
C(\R^m)$, there are  $g_1, \ldots , g_{2m+1}$ in $C(\R)$ 
such that $f=\sum_{q=1}^{2m+1} g_q \circ \phi_q$.
\end{description}

First some preliminary definitions.
Let $K_n^i$ be 
\[ \hfill \{(x_1,x_2,\cdots, x_m):x_i\in [-n-2,-n]\cup [n,n+2],\ x_j\in [-n-2,n+2] \text{ for } j\neq
 i\},\]
and let  $\mathcal{K}=
 \{K_n=\bigcup_{i=1}^m K_n^i: n\in \N\cup\{0\}\}$.

 For each $n$, the image of $K_n$ under $\phi^q$ is
 $\{[n-1,m(n+2)+1
 ]: n\in\N\cup\{0\}\}$ which is
 a locally finite collection of subsets of $\R$.

\medskip

Next we inductively define a sequence of continuous functions $\alpha_n$ on $\R^m$ for $n\in \N\cup\{0\}$, as follows:
\begin{description}
  \item[Base step:] $\alpha_0(\mathbf{x})=1$ for $\mathbf{x}\in
  [-1,1]^m$, $\alpha_0(\mathbf{x})=0$ for $\mathbf{x} \in \R^m\setminus K_0$.
  \item[Inductive step:] $\alpha_n(\mathbf{x})=1-\alpha_{n-1}(\mathbf{x})$ for $\mathbf{x} \in K_n\cap
  K_{n-1}$, $\alpha_n(\mathbf{x})=0$ for $\mathbf{x}\in \R^m\setminus K_n$.
\end{description}

\medskip

To prove the Claim, take  any $f\in C(\R^m)$. Then
$f(\mathbf{x})=\sum_{i=0}^{\infty}\alpha_i(\mathbf{x})\cdot
f(\mathbf{x})$. Also $\alpha_i(\mathbf{x})\cdot
f(\mathbf{x})=0$ if $\mathbf{x} \notin K_i$.


From the Claim in the previous Step, for each $i\in \N\cup\{0\}$, there exist
continuous functions $g_1^i, \ldots , q_{2m+1}^i$ such that
$\alpha_i(\mathbf{x})\cdot f(\mathbf{x})=\sum_{q=1}^{2m+1}
g^i_q(\phi^q(\mathbf{x}))$.

Then let $g_q=\sum_{i=0}^{\infty}g^i_q$. This function is
well-defined and continuous because   $\{x:g^i_q(x)\neq 0\}\subseteq
[i-1,m(i+2)+1 ]$, which means there are only finitely many $i$ with 
$g^{i}_q(x)\neq 0$ for  each $x\in \R$.

 Then we have 
\[f(\mathbf{x})=\sum_{i=0}^{\infty}\alpha_i(\mathbf{x})\cdot
f(\mathbf{x})=\sum_{i=0}^{\infty}\sum_{q=1}^{2m+1}
g_q^i(\phi^q(\mathbf{x}))=\sum_{q=1}^{2m+1}
g_q(\phi^q(\mathbf{x})),\]
--- as claimed.
\end{proof}

\begin{thm}\label{char}
Let $X$ be a Tychonoff space. Then the following are equivalent:
\begin{itemize}
\item[(1)] some power of $X$ has a finite basic family;
\item[(2)] for every $m, n \in \N$, there is an $r \in \N$ and $\psi^{p q}$ from $C(X,\R^n)$, for $q=1,\ldots, r$ and $p=1, \ldots , m$, such that every $f \in C(X^m, \R^n)$ can be written 
\[f(x_1, \ldots , x_m) = \sum_{q=1}^r g_q \left( \sum_{p=1}^m \psi^{pq} (x_p) \right),\]
for some $g_1, \ldots , g_r$ in $C(\R^n,\R^n)$;
\item[(3)] $X$ is a locally compact, finite dimensional separable metric space, or equivalently, homeomorphic to a closed subspace of Euclidean space.
\end{itemize}
\end{thm}

\begin{proof} It was shown in \cite{FG1} that a Tychonoff space has a finite basic family if and only if it is a locally compact, finite dimensional separable metrizable space. Hence (1) implies (3), and (2) implies (1).

Now suppose (3) holds and $X$ is a locally compact, finite dimensional separable metric space. Fix $m$. Then $X$ is (homeomorphic to) a closed subspace of some $\R^{\ell}$. We establish (2)  when $n=1$. The general case follows easily by working co--ordinatewise.

According to Theorem~\ref{ElemRm} there exist $\psi^{pq}$ for 
$p=1,2,\ldots, \ell m$ and $q=1,2, \ldots, 2\ell m+1$ such that any $f\in
C(\R^{\ell m})$ can be written as $f(x_1,\ldots,x_{\ell m})=\sum_{q=1}^{2 \ell m+1}g_q(\sum_{p=1}^{ \ell m}\psi^{pq}(x_p))$ for some  $g_q\in C(\R)$.

Let $r=2 \ell m +1$. Let $\Psi^{pq}=\sum_{i=1+(p-1)m}^{m+(p-1)m} \psi^{iq}$ for
$p=1 , \ldots , m$ and $q=1,\ldots, r$. Since $X$ is a closed subset of $\R^\ell$, any continuous function on $X$ can be continuously extended to $\R^\ell$. Then
$\{\Psi^{pq}\restriction {X}: p=1,\ldots, m,$ and $q=1, \ldots, r\}$ are as required.
\end{proof}

Note that from Theorem~\ref{char}(2) it follows that every continuous function of three complex variables can be written as a superposition of addition and continuous functions of one complex varaiable.

\end{document}